\documentclass{article}

\usepackage{url}
\usepackage{psfrag}
\usepackage{graphicx}
\usepackage{comment}
\usepackage{array}
\usepackage{amsmath,amsfonts,amssymb,amsxtra,amsthm}
\usepackage{mathrsfs}
\usepackage{verbatim}
\usepackage{xy}
\xyoption{all}
\usepackage{stmaryrd}
\usepackage{appendix}
\usepackage{xr}
\usepackage{fancyhdr}
\usepackage{pinlabel}\DeclareMathOperator{\outhom}{OutHom}

\DeclareMathOperator{\stab}{Stab}

\DeclareMathOperator{\Aut}{Aut}

\DeclareMathOperator{\Hom}{Hom}

\DeclareMathOperator{\Ker}{Ker} 

\DeclareMathOperator{\Mod}{Mod}

\DeclareMathOperator{\Out}{Out}

\DeclareMathOperator{\stabker}{\underrightarrow{\Ker}}

\DeclareMathOperator{\trace}{tr}

\DeclareMathOperator{\tr}{tr}

\def\sl2c{\ensuremath{{SL}(2,\mathbb{C})}}
\def\psl2c{\ensuremath{{PSL}(2,\mathbb{C})}}

\def\t1sl2c{{\mathfrak sl}_2\mathbb{C}}
\def\free{\ensuremath{\mathbb{F}}}
\def\zee{\mathbb{Z}}

\def\define{\raisebox{0.3pt}{\ensuremath{:}}\negthinspace\negthinspace=}

\def\fontop#1{\stackrel{#1}{\longrightarrow}}

\def\immerses{\looparrowright}
\def\onto{\twoheadrightarrow}
\def\into{\hookrightarrow}

\def\ncl#1{\mathord{\langle}\mskip -4mu plus 0mu minus 0mu \mathord{\langle}#1\mathord{\rangle}\mskip -4mu plus 0mu minus 0mu \mathord{\rangle}}

\makeatletter
\newcommand{\Doubletwo}[3]{ \left\{ #1 \mid #3\right\} }
\newcommand{\Doubleone}[1]{ \left\{ #1 \right\} }
\newcommand{\Doublethr}[3]{ \left\{ #1 \right\}_{#3} }
\newcommand{\set}[1]{%
\@ifnextchar:{\Doubletwo{#1}}{\@ifnextchar_{\Doublethr{#1}}{\Doubleone{#1}}}%
}
\makeatother

\makeatletter
\newcommand{\grouptwo}[3]{ \langle #1 \mid #3\rangle }
\newcommand{\groupone}[1]{ \langle #1 \rangle }
\newcommand{\group}[1]{%
\@ifnextchar:{\grouptwo{#1}}{\groupone{#1}}%
}
\makeatother

\newtheorem*{theorem*}{Theorem}
\newtheorem{theorem}{Theorem}
\newtheorem{lemma}[theorem]{Lemma}
\newtheorem{example}[theorem]{Example}
\newtheorem{corollary}[theorem]{Corollary}

\newtheorem{question}{Question}

\theoremstyle{definition}
\newtheorem{definition}[theorem]{Definition}

\theoremstyle{remark}
\newtheorem{remark}[theorem]{Remark}

\newcommand{\mnote}[1]{}

\title{Simple loop conjecture for limit groups}

\author{Larsen Louder}

\begin{document}
\maketitle

\begin{abstract}
  There are noninjective maps from surface groups to limit groups that
  don't kill any simple closed curves. As a corollary, there are
  noninjective all-loxodromic representations of surface groups to
  $\sl2c$ that don't kill any simple closed curves, answering a
  question of Minsky. There are also examples, for any $k$, of
  noninjective all-loxodromic representations of surface groups
  killing no curves with self intersection number at most $k$.
\end{abstract}

\section{Introduction}

\fancypagestyle{firststyle}
{
   \fancyhf{}
   \fancyhead[L]{Accepted for publication in Israel Journal of Mathematics}
}

\thispagestyle{firststyle}

\par The simple loop conjecture says that non-injective maps from
closed orientable surface groups to three-manifold groups kill simple
closed curves. A related question, due to Minsky and motivated by the
case of a hyperbolic three-manifold, is whether or not the same holds
for maps to $\sl2c$. We give examples in all genera showing that this
is not the case.

\begin{question}[Delzant]
  Does the simple loop conjecture hold for limit groups? 
\end{question}

\par The negative answer implies a negative answer to Minsky's question
since limit groups embed in $\sl2c$. Jason
Manning and Daryl Cooper have constructed, using different methods,
examples of nonfaithful representations of surface groups which kill
no simple closed curves, but whose images contain parabolic
elements~\cite{manningnonfaithful}.

\begin{theorem}
  \label{maintheorem}
  The following hold:
  \begin{itemize}
  \item Let $S$ be a compact orientable surface of genus at least
    three.  Given $k$ there is a limit group $L$ and a
    non-$k$-pinching noninjective map $\pi_1(S)\to L$.
  \item Let $M$ be the figure-eight knot complement. Let $S$ be the
    compact orientable surface of genus two. Given $k$ there is a
    $\pi_1(M)$-limit group $L$ and a noninjective non-$k$-pinching map
    $\pi_1(S)\to L$.
\end{itemize}
\end{theorem}

\begin{corollary}
  \label{maintheorem2}
  Given $k$ and a compact orientable surface $S$ of genus at least
  two, then there is a nonfaithful, non-$k$-pinching, all-loxodromic
  representation $\pi_1(S)\to \sl2c$.
\end{corollary}

\begin{remark}
  In the first bullet of Theorem~\ref{maintheorem} $L$ must depend on
  $k$: Every noninjective map $\pi_1(S)\to L$ factors through, up to
  precomposition by an element of $\Mod(S)$, one of finitely many fixed limit
  groups. If $g$ has self-intersection number $k$ then, for all
  $\varphi\in\Mod(S)$, $\varphi(g)$ has self intersection number $k$, hence
  every map $\pi_1(S)\to L$ kills an element of definite intersection number.
\end{remark}

As a corollary of the proof of Theorem~\ref{maintheorem} we also have the
following:

\begin{corollary}
  \label{cor:elementary} The property of admitting only $k$-pinching
  noninjective maps of surface groups is not closed under elementary
  equivalence.
\end{corollary}

Sela characterizes the finitely generated groups elementarily
equivalent to a free group $\free$ as \emph{hyperbolic towers} over
$\free$. Some of the examples of limit groups constructed to prove the
second bullet of Theorem~\ref{maintheorem} are towers over a free
group and admit non-$k$-pinching maps of closed surfaces.  On the
other hand, every (noninjective!) map from a closed surface to a free
group kills a simple closed curve.

\begin{question}
  \label{question::genus2}
  Does every noninjective map from the fundamental group of the closed
  genus two surface to a limit group kill a simple closed curve?
\end{question}

\subsection*{Acknowledgments}

I would like to thank Jason Manning for pointing out Hempel's theorem,
Matt Stover for suggesting the figure-eight knot complement, Chlo\'e
Perin for telling me Delzant's question, Mladen Bestvina for sharing
with me his and Mark Feighn's notes on negligible sets in the free
group, the referee for finding a number of places where the paper
could be improved, and Juan Souto for yelling at
me.

\section{Background}

\begin{definition}[Intersection number/Pinching]
  Let $S$ be a surface, and $c$ a free homotopy class of curves on
  $S$. The self intersection number of $c$ is the minimal number of
  points of self intersection of curves without triple points in the
  homotopy class $c$. We call a free homotopy class of curves $c$
  $k$-simple if its self-intersection number is at most $k$. We say
  simple rather than $0$-simple.

  Let $S$ be a surface, possibly with boundary.  A continuous map
  $S\to X$ is $k$-pinching if the image of some $k$-simple curve is
  homotopically trivial in $X$.  A map $S\to X$ is non-$k$-pinching
  if it is not $k$-pinching. The same terms will be used for
  homomorphisms of surface groups.

  We write pinching rather than $0$-pinching.
\end{definition}

\begin{definition}[Limit group]
  Let $\Gamma$ be a torsion free (toral relatively) hyperbolic
  group. A finitely generated group $G$ is \emph{fully residually
    $\Gamma$}, or a $\Gamma$-\emph{limit group}, if, for each finite
  $S\subset G$, there is a homomorphism $f\colon G\to \Gamma$
  embedding $S$. If we write \emph{limit group} without the modifier
  we mean an $\free$-limit group. Let $G$ be a finitely generated
  group, and let $f_n\colon G\to \free$ be a sequence of maps from $G$
  to a free group. We say that $f_n$ is \emph{stable} if, for all $g$,
  there is $n_g$ such that $f_m(g)=1$ for $m>n_g$ or $f_m(g)\neq 1$
  for $m>n_g$. The stable kernel $\stabker(f_n)$ of a stable sequence
  is the collection of elements whose images are eventually
  trivial. Alternatively, by~\cite{sela::dgog1}, $L$ is a limit group
  if and only if $L$ is the quotient of a finitely generated group by
  the stable kernel of a stable sequence, i.e., $L=G/\stabker(f_n)$
  with $f_n$ stable. In this case we say that $f_n$ converges to
  $L$. Eventually $f_n$ factors through the quotient map
  $\pi\colon G\onto L$\cite[Theorem~1.17(iii)]{sela::dgog7}.
\end{definition}

Let $X$ be a space, and let $Y\subset X$. If $f\colon Y\to Z$, denote
by $X/f$ the space $(X\sqcup Z)/(y\sim f(y))$.

\begin{definition}[Graph of surfaces]
  We call a space $X$ a \emph{graph of surfaces} if there are
  connected spaces $X_1,\dotsc,X_n$, (``vertex spaces'') compact
  surfaces with boundary $\Sigma_1,\dotsc,\Sigma_m$, either with Euler
  characteristic at most $-2$ or homeomorphic to a torus with one
  boundary component, and annuli $A_1,\dotsc,A_l$, such that
  $X=(\sqcup X_i\sqcup\Sigma_j\sqcup A_k)/\sim$, where $\sim$ is the
  equivalence relation generated by a map
  $\sqcup\partial\Sigma_j\sqcup\partial A_k\to\sqcup X_i$ with the
  property that each boundary component is mapped to a
  non-nullhomotopic loop in $\sqcup X_i$, and such that for at least
  one boundary component $C$ of an annulus $A$ then the image of $c$
  is maximal abelian in $X\setminus A^{\circ}$.
\end{definition}

Let $X$ be a graph of surfaces and suppose $\pi_1(Y)$ is a
$\Gamma$-limit group.  A map $X\to Y$ embedding the fundamental group
of each $X_i$ and sending each $\pi_1(\Sigma_j)$ to a nonabelian
subgroup of $\pi_1(Y)$ is \emph{strict}.
  
\begin{theorem}[{\cite[Theorem~1.31]{sela::dgog7}}]
  If there is a strict map $X\to Y$, and $\pi_1(X)$ is f.g., then
  $\pi_1(X)$ is a $\Gamma$--limit group.
\end{theorem}

The proof in the toral relatively hyperbolic case is the same as in
the free case. 

If, in a graph of surfaces, there is only one vertex space $X_1$ and
there is a strict retraction $X\to X_1$ then $X$ is a \emph{floor}
over $X_1$. If $X$ admits a filtration $Y=Y_1\subset
Y_2\subset\dotsb\subset Y_n=X$ such that $\pi_1(Y_1)$ is a torsion
free hyperbolic group and $Y_i$ is a floor over $Y_{i-1}$ then $X$ is
a \emph{tower} over $Y$. If a floor has no annuli it is a
\emph{hyperbolic floor}, and if a tower only has hyperbolic floors
then it is a \emph{hyperbolic tower}

If $G<H$ and there are spaces $Y$ and $X$ such that $X$ is a tower
over $Y$, and identifications $H=\pi_1(X)$ and $G=\pi_1(Y)$ such that
the inclusion $Y\into X$ induces the inclusion map $G\into H$ then we
say that $H$ is a tower over $G$.

\begin{theorem}[\cite{sela::dgog6,sela::dgog7}]
  \label{sela::towers}
  If $X$ is a hyperbolic tower over $Y$ and $\pi_1(Y)$ is a
  non-trivial torsion free hyperbolic group then $\pi_1(X)$ and
  $\pi_1(Y)$ are elementarily equivalent.
\end{theorem}

\section{Outer homomorphisms/Graphs}

We assume familiarity with Stallings folding~\cite{stallings0} and
immersions of graphs.

\begin{definition}
  Let $S$ be a compact surface, possibly with boundary. The
  \emph{modular group} of $S$, $\Mod(S)$, is the group of
  self-homeomorphisms of $S$ fixing the boundary of $S$, up to homotopy. Clearly
  $\Mod(S)<\Out(\pi_1(S))$.

  Let $G$ and $H$ be groups. The set of \emph{outer homomorphisms}
  $\outhom(G,H)$ is the set of homomorphisms $f\colon G\to H$ modulo
  conjugation in $H$. Precomposition of homomorphisms by elements of
  $\Out(G)$ gives a well defined action of $\Out(G)$ on
  $\outhom(G,H)$.

  We say that two homomorphisms $f,g\colon\pi_1(S)\to G$ are
  $\Mod(S)$--inequivalent if they represent different equivalence
  classes in $\outhom(\pi_1(S))/\Mod(S)$.
\end{definition}

\begin{definition}[Track]
  Let $S$ be a surface with boundary. A \emph{track} in $S$ is a one-manifold
  with boundary $T\subset S$ such that $\partial T\subset\partial S$. A
  \emph{labeled track} in $S$ is a track in $S$ such that each connected
  component of $T$ is equipped with a transverse orientation and a letter from
  some alphabet $a,b,\dotsc$. 
\end{definition}

\par If $c$ is an oriented curve in $S$ transverse to a labeled track
$T$ there is a (conjugacy class of) word(s) $w_T(c)$ in the alphabet
$a,b,\dotsc$ obtained by reading the labels from $T$, with orientations, as $c$
is traversed. Let $f\colon \pi_1(S)\to\langle a,b\rangle$ be a morphism of free
groups. Represent $\langle a,b\rangle$ as the fundamental group of a rose $R$
with two oriented edges $a$ and $b$. There is a continuous map $r\colon S\to R$,
representing the outer homomorphism represented by $f$, which we may assume is
transverse to midpoints $m_a$ and $m_b$ of the edges $a$ and $b$. Let
$T=r^{-1}(m_a,m_b)$, and give $T$ the structure of a labeled track induced by
the orientations of $a$ and $b$, with the labels $a$ and $b$ assigned according
to whether a component is mapped to $m_a$ or $m_b$. If $c$ is a simple closed
curve in $S$ we may assume without loss of generality that the word $w_T(c)$ is
cyclically reduced, and that words representing boundary components are
cyclically reduced as well. (Note that we cannot do this for all curves
simultaneously.)

\par Let $T$ be an oriented track in a surface with boundary $S$, and
let $\Gamma_T$ be the oriented graph with vertex set $S\setminus T$
with an oriented labeled edge for each connected component of
$T$. There are natural maps (defined up to homotopy) $S\to \Gamma_T$
and $\Gamma_T\to R$.

\begin{definition}
  If $w$ is cyclically reduced, say that $u$ is a piece of $w$ if $w$
  can be written without cancellations as $u_2vu_1$ and $u_2'v'u_1'$
  with $u=(u_1u_2)^{\pm 1}=(u_1'u_2')^{\pm 1}$ and $u_2\neq u_2'$. For
  $w$ cyclically reduced, define
  \[
  o(w)=\max\{\vert u\vert/\vert w'\vert\mid u \mbox{
    is a piece of }w\}
  \]
  and for arbitrary $w$ define $o(w)=o(w')$, where $w'$ is any
  cyclically reduced conjugate of $w$.
 \end{definition}

\section{Surfaces with boundary}

In this section we show that there are many non-injective
non-$k$-pinching maps from surfaces with boundary to free groups, hyperbolic groups
and the fundamental group of the figure-eight knot complement. In the next
section we extend them to closed surfaces.

\begin{theorem}
  \label{lem:nonpinchingsequence}
  Let $S$ be a four-times punctured sphere, and let $b_1,\dotsc,b_4$
  be generators for four nonconjugate boundary subgroups. Let $f_n\colon
  \pi_1(S)\onto F_2=\langle a,b\rangle$ be a sequence of surjective pairwise
  $\Mod(S)$-inequivalent homomorphisms such that the images $f_n(b_i)$
  are nontrivial, $f_n(b_i)$ and $f_n(b_j)$ have nonconjugate
  centralizers for $i\neq j$, and $o(f_n(b_i))\fontop{n\to\infty} 0$ if $\vert
  f_n(b_i)\vert\to\infty$. Then $f_n$ is eventually non-$k$-pinching
  for any $k$.
\end{theorem}

\mnote{strictly speaking precomposition by an element of mod(s) does not in general preserve the hypotheses on fn in particular the fact that o of boundary components to 0 might not remain true after precomposition. this probably comes from a problem with my definition of o(w)}

\begin{proof}
  Nonconjugacy of centralizers of images of boundary components
  guarantees that each of the maps $f_n$ is non-pinching.

  Precomposition by elements of $\Mod(S)$ doesn't change nonconjugacy
  of centralizers of images of boundary components, surjectivity of
  the maps $f_n$, or the conjugacy classes of images of elements of
  boundary components. Since there are only finitely many $k$-simple
  closed curves up to the action of $\Mod(S)$, hence if $f_n$ is
  eventually $k$-pinching then we may assume, by passing to a
  subsequence and precomposing by elements of $\Mod(S)$, that each
  $f_n$ kills some fixed element $g\in\pi_1(S)$. Pass to a stable
  subsequence, also denoted $f_n$, converging to a proper limit
  quotient $L$ of $\pi_1(S)$, that is, $f_n=\overline f_n\circ\pi$,
  where $\pi\colon\pi_1(S)\onto L$ is the quotient map~\cite[Proof of
  Theorem~4.6]{sela::dgog1}. The sequence $\overline f_n$ converges to
  a faithful stable action of $L$ on a real tree $T$, and we use the
  Rips machine to analyze the action of $L$ on $T$.

  Let $X$ be a finite complex with fundamental group $L$ with the
  property that for each $i$ the conjugacy class of $\pi(b_i)$ is
  represented by a reduced edge path $p_i$ in the one skeleton, and
  such that $p_i$ and $p_j$ have disjoint images for $i\neq j$. Let
  $\tilde X\to T$ be a resolving map as
  in~\cite[Proposition~5.3]{bf:sa}, however choose the map to send
  each lift of $p_i$ homeomorphically to the axis of the appropriate
  conjugate of $b_i$ if $b_i$ acts hyperbolically, and to a point if
  $b_i$ acts elliptically, in $T$.

  Then $X$ has the structure of a band
  complex~\cite[Definition~5.1]{bf:sa} with transverse measure
  $\mu$. Let $Y\subset X$ be the union of bands associated to
  $X$. Each $p_i$ may be written as a composition of paths
  $v_0h_0\dotsb v_{k-1}h_{k-1}$, where each $h_j$ is a horizontal path
  in $Y$, $\mu(v_j)=0$, the translation length of $b_i$ in $T$ is
  $\mu(h_1)+\dotsb+\mu(h_k)$, each lift of $h_j$ in $\tilde X$ is
  embedded in $T$ under the resolving map, and for every lift of
  $h_jv_{j+1}h_{j+1}$ the resolving map sends the lifts of $h_j$ and
  $h_{j+1}$ to arcs intersecting in a point. We call such a path
  \emph{immersive}. If $p$ is an immersive path and $X$ and $X'$
  differ by Rips moves then the associated path $p'$ in $X'$ is
  immersive as well.

  The output of running the Rips machine on $X$ is a band complex,
  which we also call $X$, such that each minimal component of $X$ is
  either toral, thin, or surface type. We leave the following claim to
  the reader as an exercise in Gromov-Hausdorff convergence: Let
  $f_n\colon L\to\free$ be a sequence of homomorphisms converging to
  $L$, and suppose $p$ is an immersive edge path in the band complex
  $X$ associated to the action of $L$ on the limiting tree $T$. If
  $p_i$ meets (has positive measure in) a minimal component of $Y$, or
  if $p_i$ meets a simplicial component with nontrivial fundamental
  group, then $o(f_n(b_i))>\epsilon>0$ for some fixed
  $\epsilon$. See~\cite[{\S7}]{bf::lg}.

  \mnote{contradicting the fact that fn is nonpinching. it would be
    clearer to say that this gives a contradiction to nonconjugacy of
    images of boundary components, which is the reason fn is
    nonpinching, but i haven't mentioned this yet.}

  Suppose $Y$ has a minimal component $Z$. If $Z$ is thin-type then
  $L$ splits over the trivial group relative to $\pi(b_i)$,
  contradicting the fact that $f_n$ is non-pinching. If $Z$ is surface
  type then $L$ is free of rank two and $\pi(b_i)$ are all contained
  in the commutator subgroup, contradicting the fact that they
  generate in first homology, and if $Z$ is toral type then $L$ is
  either $\zee*\zee^2$ or $\zee^3$. \mnote{the second case is rulded
    out by nonconjugace: isn't it rather by the fact that the images
    generate in homology? in fact probably all the cases with a toral
    component can be excluded by this argument, just as for a surface
    component} The second case is ruled out by nonconjugacy of
  centralizers of images of boundary subgroups, and in the first case
  $L$ again splits freely relative to the images of boundary subgroups
  and the map $\pi_1(S)\to L$ kills a simple closed curve.

  \mnote{all compoennts: need to justify that all the components of Y
    are crossed by a pi, possibly using that the b's generate} Thus
  all components of $Y$ are simplicial and carry the trivial
  fundamental group, and since $o(f_n(b_i))\to 0$ whenever $b_i$ acts
  hyperbolically in $T$, each $p_i$ meets any given simplicial
  component of $Y$ either once or not at all. Furthermore, each
  component of $Y$ gives a splitting of $L$ over the trivial
  group. Replace $X$ by the band complex obtained by collapsing each
  leaf to a point. Leaves of $Y$ carry the trivial group, the quotient
  map is an isomorphism on $\pi_1$, therefore we may assume that that
  each component of $Y$ is an interval. Furthermore, again without
  changing the fundamental group, we may assume that the complementary
  components of the interior of $Y$ are either points, circles or
  tori.
  
  \mnote{the justification that all the bi's are hyperbolic seems
    insufficient. what is clear is that they cannot all by in
    conjugates of the xsquared corresponding to the torus sinc ethey
    generate in homology and that if one is hyperbolic then at least
    two of them are hyperbolic. the case where only two of them are
    hyperbolic corresponds to the case where there is only one arc:
    need to also consider the one arc case. DO I??} If a complementary
  component is a torus then $L\cong \zee^2*\zee$ and all $b_i$ are
  hyperbolic in $T$, otherwise a simple closed curve in $S$ has
  trivial image, contradicting the fact that $f_n$ is nonpinching. Let
  $m$ be a point in the interior of $Y$, and represent the map
  $\pi_1(S)\to L$ by a continuous map $g\colon S\to X$ extending the
  induced maps $p_i$ on boundary components. The map $g$ may be chosen
  so that $g^{-1}(m)$ consists of at most two arcs connecting distinct
  pairs of boundary components of $S$, and induces a measured
  foliation $\Lambda$ on $S$. If there is one arc then $f_n(b_i)$ and
  $f_n(b_j)$ have conjugate centralizers for some $i\neq j$ and
  sufficiently large $n$, thus there are two arcs, but the complement
  of the two arcs is homotopy equivalent to a circle, but there are no
  maps $\zee\onto\zee^2$.

  Thus we may assume that $X$ is a graph of rank two. By collapsing we
  may arrange that no edge not contained in $Y$ is embedded. Clearly
  $X$ is either a barbell graph, a theta graph, or a rose. See
  Figure~\ref{fig::xoptions}, where $Y$ is represented by bold
  subgraphs.

\begin{figure}[ht]
  \centerline{
    \labellist
    \pinlabel $p$ at 486 55
    \pinlabel $q$ at 607 56
    \endlabellist
    \includegraphics[width=.7\textwidth]{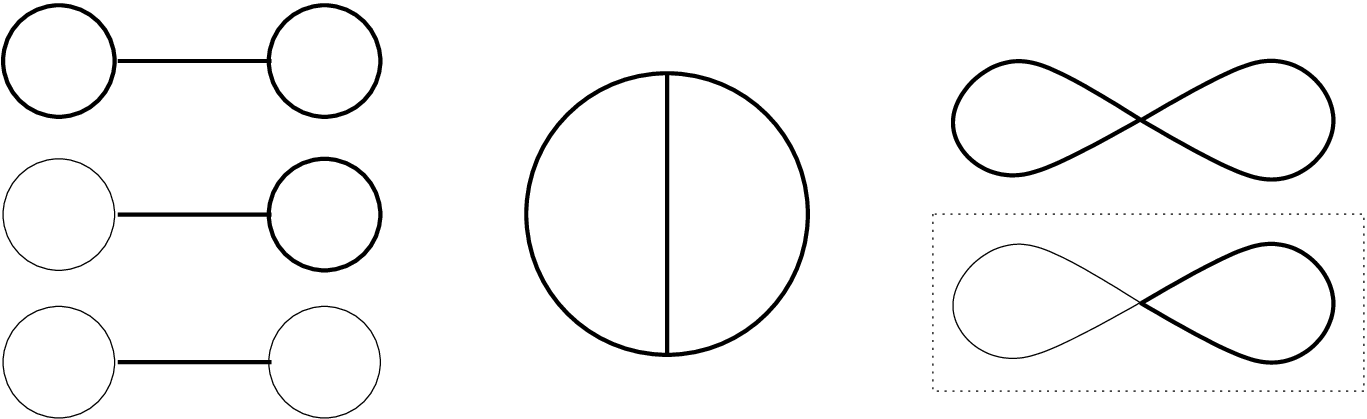}
  }
\caption{Possibilities for $X$.}
\label{fig::xoptions}
\end{figure}

  If $X$ is a barbell graph then the middle edge must be a component
  of $Y$. Each $b_i$ crosses each edge of $Y$ at most one time, hence
  each $b_i$ is contained in one of two embedded subgraphs, each of
  which is a circle, again contradicting the condition that the $b_i$
  have pairwise nonconjugate centralizers under all $f_n$.
  
  Suppose $X$ is a theta graph. Then all edges are components of
  $Y$. The theta graph has only three distinct embedded paths,
  contradicting the condition on nonconjugate centralizers of images,
  hence the only possibility is that $X$ is a rose.

  Suppose $X$ is a rose and that both edges are components of $Y$, and
  let $c$ and $d$ be the edges. There are four distinct closed
  embedded (except at vertices) paths $p$, $q$, $pq$, and $pQ$ in
  $X$. Since the $b_i$ have pairwise nonconjugate centralizers and sum
  to zero in $\zee_2$ homology of $S$, we must have
  \[
  0=\left[p+q+pq+pQ\right]=\left[p+q\right]
  \] 
  in $\mathrm{H}_1(X,\mathbb{Z}_2)$, which is a contradiction.  Thus
  $X$ is a rose and $Y$ consists of one edge. 

  Suppose some $b_i$ acts elliptically in $T$. Then $b_i$ is sent to a
  path in $X$ which misses $Y$, and since each $b_i$ crosses $Y$ at
  most one time, and since there are only four boundary components,
  some other $b_j$, $j\neq i$, is also elliptic in $T$, contradicting
  the hypothesis that $f_n(b_i)$ have nonconjugate centralizers. Thus
  each $b_i$ crosses $Y$ exactly once. As above, let $g\colon S\to X$
  be a continuous map representing $\pi$, and let $c$ be the
  non-boundary parallel simple closed curve in the complement of
  $g^{-1}(m)$. The simple closed curve $c$ is elliptic in $T$, and we
  may assume that $g(c)$ does not meet the interior of $Y$.
  
  Let $T_n$ be the labeled track in $S$ associated to the map
  $f_n$. As above, we may assume that $w_{T_n}(b_i)$ and $w_{T_n}(c)$
  are cyclically reduced\mnote{this is true only if the definition of
    o is changed as suggested as above}.  Since $f_n$ is nonpinching,
  we may assume that no component of $T_n$ is a circle, and since $c$
  is elliptic in $T$, and the $b_i$ are not elliptic in $T$, for large
  $n$, the track $T_n$ has no subarcs connecting $c$ to itself.

  Let $A$ be an annular regular neighborhood of $c$. After perhaps
  precomposing by a Dehn twist supported in the annulus $A$, and
  passing to a subsequence of $f_n$, the track $T_n$ has the form
  illustrated in the middle of
  Figure~\ref{fig::pairsofpantsdecomp}. In particular, up to homotopy,
  $f_n$ factors as
  \[
    f_n=h_n\circ\sigma\circ\tau^{k_n}
  \]
  where $h_n\colon \Gamma\to R$ is a morphism from some \emph{fixed}
  graph $\Gamma$, illustrated at the right of
  Figure~\ref{fig::pairsofpantsdecomp}, to the rose, given by
  assigning elements $E_n$, etc., of $\langle a,b\rangle$ to each edge
  $E$, etc., and $\tau^{k_n}$ is a power of the Dehn twist in $A$.
  Furthermore, the concatenations of words $E_nB_nG_n$,
  $A_nG_n^{-1}F_n^{-1}$, $F_nB_n^{-1}$, and $A_nE_n$ are all
  cyclically reduced and represent the images of the boundary
  components. The subgraphs with labels $E_n,F_n,G_n,A_n$ and
  $E_n,F_n,G_n,B_n$ immerse in $R$ under $h_n$. Note that in general
  $E_n$, $F_n$ or $G_n$ may be trivial. The $f_n$ all kill some
  fixed element $w$, so $h_n\circ\sigma$ kills $\tau^{-k_n}(w)$, but,
  for sufficiently large $k_n$, $\tau^{-k_n}(w)$ takes no turns in
  $\Gamma$ (in the case we are considering, $A^{-1}B$) which are
  folded by $h_n$. Thus there are only finitely many $\{{k_n}\}$. By
  passing to a further subsequence and taking a fixed Dehn twist we
  may assume that $f_n$ factors as $h_n\circ\sigma$.

  \begin{figure}[ht]
  \labellist
  \pinlabel $\tau^{k_n}$ [b] at 359 120
  \pinlabel $\sigma$ [b] at 771 120
  \pinlabel $A_n$ [r] at 927 116
  \pinlabel $B_n$ [tr] at 1136 222
  \pinlabel $E_n$ [br] at 858 162
  \pinlabel $F_n$ [l] at 1029 138
  \pinlabel $G_n$ [tl] at 1013 33
  \pinlabel $\Gamma$ at 873 252
  \endlabellist
  \centerline{
  \includegraphics[width=\textwidth]{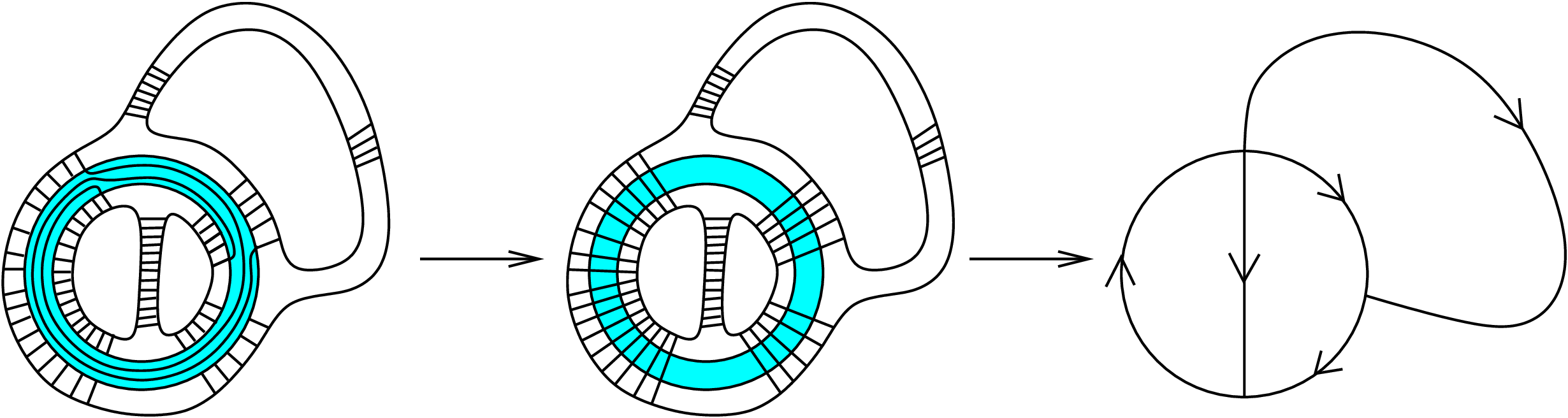}
  }
  \caption{Each track $T_n$ has the same form.}
\label{fig::pairsofpantsdecomp}
\end{figure}
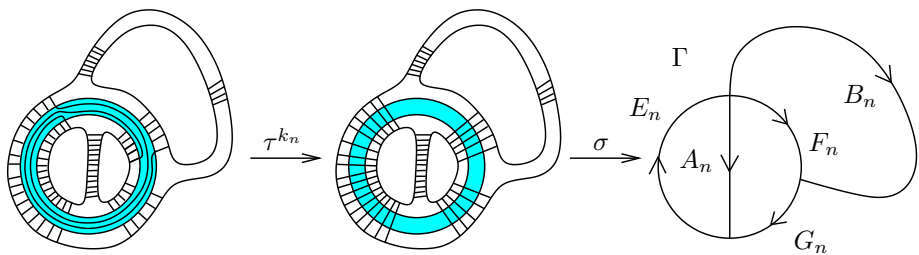

Without loss of generality, we may write $A_n$ and $B_n$ as reduced
products $A_n=C_nL_n$, $B_n=C_nM_n$. Suppose that, without loss,
$\vert B_n\vert/\vert A_n\vert\to\infty$. Then the subgroup
corresponding to the subgraph spanned by edges labeled $A,E,F,G$ acts
elliptically in $T$, hence has abelian image under all $f_n$,
contradicting again nonconjugacy of centralizers. Thus there are
constants $\infty>K,U>0$ such that $K\leq \vert B_n\vert/\vert
A_n\vert\leq U$ for all $n$. Furthermore, since $C_n$ corresponds to
the subset of $S$ mapping to $Y$, we have
\begin{equation}
  \vert C_n\vert/\max\{\vert E_n\vert,\vert F_n\vert,\vert G_n\vert,\vert L_n\vert,\vert M_n\vert\}\to\infty
\label{cinequality}
\end{equation}

Let $\Gamma'$ be the graph obtained by folding initial segments corresponding to
edges labeled $A$ and $B$ together, and let $\overline h_n\colon\Gamma'\to R$ be
the induced map. See Figure~\ref{factorfold}.

\begin{figure}
\labellist
\pinlabel $A_n$ [r] at 95 108
\pinlabel $B_n$ [bl] at 321 209
\pinlabel $E_n$ [br] at 19 136
\pinlabel $F_n$ [l] at 187 155
\pinlabel $G_n$ [tl] at 182 22
\pinlabel $E_n$ [br] at 492 136
\pinlabel $F_n$ [bl] at 662 155
\pinlabel $G_n$ [t] at 646 18
\pinlabel $L_n$ [r] at 575 31
\pinlabel $M_n$ [l] at 720 31
\pinlabel $C_n$ [r] at 571 108
\endlabellist
\centerline{
\includegraphics[width=.8\textwidth]{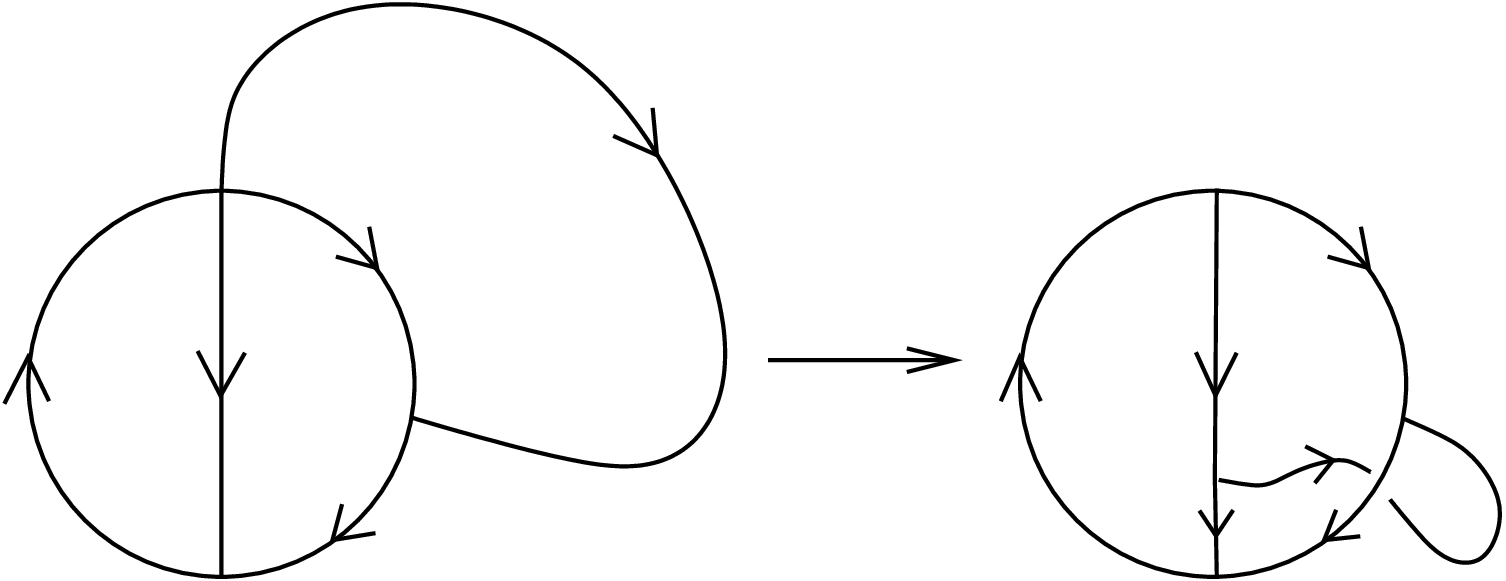}
}
\caption{The maps $h_n$ factor through a fixed quotient map $\Gamma\to\Gamma'$.}
\label{factorfold}
\end{figure}
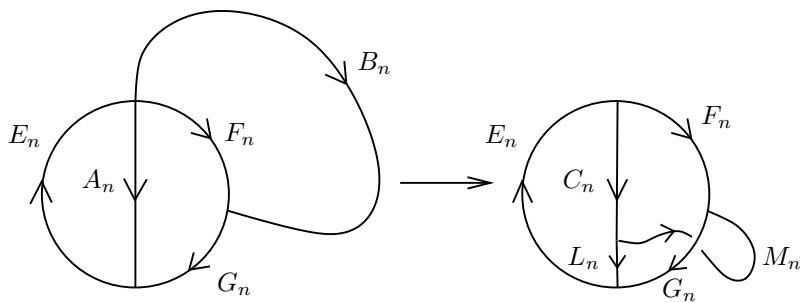

By~(\ref{cinequality}) the subgroup of $\pi_1(S)$ corresponding to the
subgraph $\Gamma'_e$ spanned by $E$, $F$, $G$, $L$, and $M$ acts
elliptically in the limiting tree. The only subgroup acting
elliptically in $T$ is infinite cyclic, therefore, for all $n$, the
subgraph $\Gamma'_e$ folds to a circle. The tails of $L_n$ and $M_n$
do not fold with $E_n$, $F_n$, and $G_n$, therefore $A_n=B_n=C_n$ and
$L_n=M_n=1$. Then there is a cyclically reduced word $K_n$ such that
$E_nF_n$, and $G_n$ are powers of $K_n$. The tail of $C_n$ doesn't
fold with the head of $E_n$ or the tail of $F_n$, and the head of
$C_n$ doens't fold with the tail of $E_n$ or the head of $F_n$, thus
$f_n$ factors through one of the immersions illustrated in
Figure~\ref{fig::immersion}. \mnote{this is clearly not true
  (surjectivity) clearly Kn An must form a basis but there is no
  reason it should be the standard one. clarify: the referee is wrong
  here. everything is arranged so that there is no folding and then it
  has to be the standard basis. ...} Such a homomorphism is surjective
only if $K_n=a$ and $A_n=B_n=b$ (or vice-versa), but then, say,
$f_n(b_1)=a^{t_n}b$, contradicting the fact that $o(f_n(b_1))\to 0$.
\end{proof}

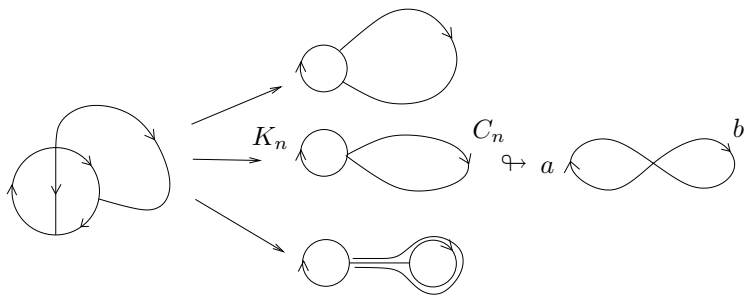
\begin{figure}
  \labellist
  \pinlabel $C_n$ [bl] at 981 313
  \pinlabel $K_n$ [br] at 614 298
  \pinlabel $a$ [r] at 1184 271
  \pinlabel $b$ [bl] at 1535 323
  \pinlabel $\immerses$ at 1082 280
  \endlabellist
  \centerline{
    \includegraphics[width=.8\textwidth]{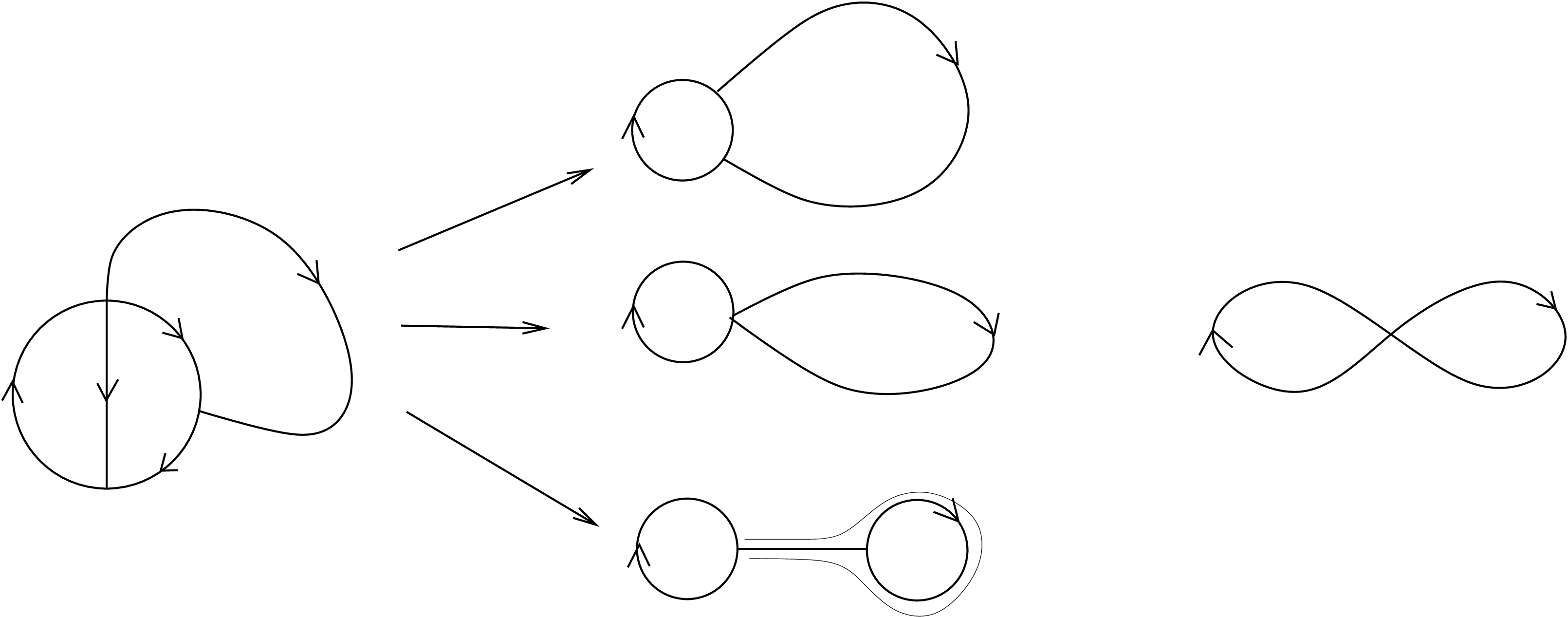}
  }
  \caption{The maps $h_n$ factor through one of three types of
    immersions. Only the center type appears, and then, by
    surjectivity of $f_n$, only if $K_n=a$ and $A_n=B_n=C_n=b$, or
    vice-versa.}
  \label{fig::immersion}
\end{figure}

\begin{example}[``$\alpha$'' maps]
  \label{alphamaps}
  Identify the fundamental group of a four-times punctured sphere with
  $\free_3=\langle x,y,z\rangle$ in such a way that $x$ $y$ and $z$
  are three boundary components, and $xyz$ is the fourth. Let $f_n$ be
  the map $x\mapsto a$, $y\mapsto b$, $z\mapsto aba^2ba^3\dotsb
  ba^nb$. Then $f_n$ satisfies the hypothesis of
  Theorem~\ref{lem:nonpinchingsequence} since the concatenation
  $f_n(x)f_n(y)f_n(z)$ is cyclically reduced and $o(f_n(z))\sim
  o(f_n(xyz))\sim 1/n$.
\end{example}

\par In \S\ref{sec::proofmain} we embed this construction in a closed surface
to find non-$k$-pinching quotients of closed surface
groups. Appropriately chosen, these will be limit groups and
embed in $\sl2c$.

\begin{lemma}
  \label{lem::killsimplesubgroup}
  Suppose $f\colon\langle x,y\rangle=F\onto L$ is a non-injective map
  such that $L$ splits as $V*_E$ or $V*_EW$ with $f(y)\in V$. Then
  $f\vert_{\langle xyx^{-1},y\rangle}$ is not injective.
\end{lemma}

\begin{proof}
  See Figure~\ref{fig::dunwoody}. Follows from Dunwoody folding
  sequences~\cite{dunwoody::folding}.
\end{proof}

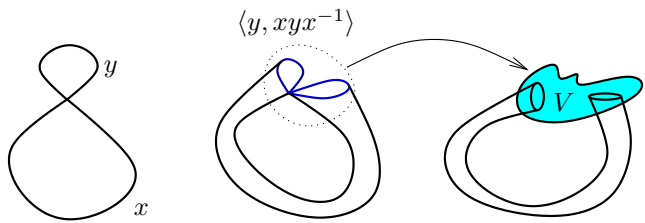
\begin{figure}[ht]
\labellist
\pinlabel $x$ [tl] at 106 21
\pinlabel $y$ [l] at 78 134
\pinlabel $\langle y,xyx^{-1}\rangle$ [b] at 253 155
\pinlabel $V$ at 487 105
\endlabellist
\centerline{
\includegraphics[width=.7\textwidth]{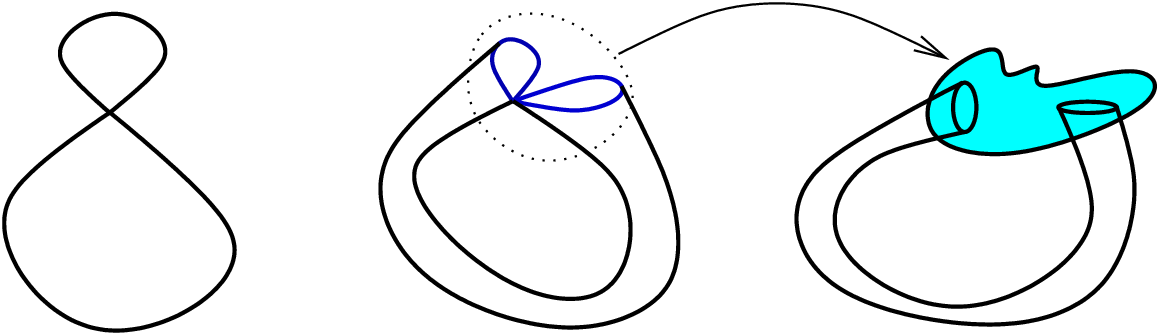}
}
\caption{Illustration for Lemma~\ref{lem::killsimplesubgroup}. The free subgroup $\langle y,xyx^{-1}\rangle$ is mapped to, but is not embedded in, $V$.}
\label{fig::dunwoody}
\end{figure}

The construction of Theorem~\ref{lem:nonpinchingsequence} fails for the
pair of pants and punctured torus since every map from these two surfaces to a free
group with nonabelian image is injective. Rather than work with a free group we
consider maps to the fundamental group of the figure-eight knot complement where
we can apply a similar technique.

\begin{lemma}
  \label{figureeightnonpinching}
  Let $M$ be the figure-eight knot complement. Let $P$ be a pair of
  pants with fundamental group $\langle x,y\rangle$, such that the
  three boundary components of $P$ are represented by $x$, $y$, and
  $xy$. There is a sequence of noninjective homomorphisms $f_n\colon
  \langle x,y\rangle \onto \pi_1(M)$ such that
  \begin{itemize}
  \item $f_n(x)$, $f_n(y)$ and $f_n(xy)$ are all nontrivial, nonparabolic,
    nonconjugate, and indivisible in $\pi_1(M)$.
  \item $f_n$ is eventually non-$k$-pinching for any $k$.
  \end{itemize}
\end{lemma}

The proof is like that of Theorem~\ref{lem:nonpinchingsequence}: The
modular group of a surface is much smaller than the automorphism group
of the underlying free group, except in the case of a punctured torus.

\begin{proof}
  Let $M$ be the figure-eight knot complement, and let $a$ and $b$ be
  the standard generators for the fundamental group illustrated in
  Figure~\ref{figureeight}. For sufficiently large $n$ the words
  $f_n(x)=a(ba^{-1})^n$, $f_n(y)=ba^{-1}$ and
  $f_n(xy)=a(ba^{-1})^{n+1}$ are nonconjugate and indivisible in
  $\pi_1(M).$ Suppose that $f_n$ is $k$-pinching for some fixed $k$. Triviality
  of the modular group of a pair of pants implies that by passing to a
  subsequence we may assume that, for all $n$,
  $f_n$ kills some fixed nontrivial element $g$ 

  Since $M$ is toral relatively hyperbolic $f_n$ converges to a
  nontrivial stable action of $\langle x,y\rangle$ on a real tree $T$,
  with $g$ in the kernel of this action
  (See~\cite{groves:relhyp}). Let $L$ be the (proper!) quotient of
  $\langle x,y\rangle$ by the kernel of the action on $T$. Then $L$ is
  a freely indecomposable limit group over $\pi_1(M)$. Since $f_n(y)$
  is constant, $b$ acts elliptically in the limiting tree, and by Rips
  machine $L$ splits nontrivially as $L=V*_EW$ or $V*_E$ over an abelian
  edge group $E$, with $y$ conjugate into $V$. By
  Lemma~\ref{lem::killsimplesubgroup}, $\langle xyx^{-1},y\rangle\to L$ is
  not injective, however, since $f_n(y)$ and $f_n(xyx^{-1})$ are
  noncommuting conjugates of an element of the fiber subgroup of
  $\pi_1(M)$ ($f_n(x)$ is not zero in homology), they generate a
  free group of rank two. Thus $L$ is free of rank two, violating the
  Hopf property for finitely generated free groups.
\end{proof}

This theorem holds for any torsion-free hyperbolic group, however 
the maps we construct are less explicit.

\begin{figure}[ht]
\labellist
\pinlabel $a$ [bl] at 68 55
\pinlabel $b$ [tr] at 51 27
\endlabellist
\centerline{\includegraphics[width=2in]{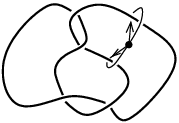}}
\caption{Generators $a$ and $b$ for the fundamental group of the
figure-eight knot complement.}
\label{figureeight}
\end{figure}

\begin{lemma}
  \label{nonpinchingpairsofpants}
  Let $G$ be a torsion free hyperbolic group and suppose that $G$
  admits a noninjective map from $\free_2=\langle x,y\rangle$ with
  nonabelian image.  There is a sequence of noninjective homomorphisms
  $f_n\colon\free_2\to G$ such that $f_n(x)$, $f_n(y)$, and $f_n(xy)$
  are nontrivial, nonconjugate, indivisible, and such that $f_n$ is
  eventually non-$k$-pinching, regarding $\free_2$ as the fundamental
  group of a pair of pants with boundary components represented by
  $x$, $y$, and $xy$.
\end{lemma}

\begin{proof}
  Let $f\colon\langle x,y\rangle\to G$ be a noninjective map of
  $\free_2$ in $G$.  Let $\varphi_n$ be a sequence of automorphisms of
  $\free_2$ fixing $\left[x,y\right]$ such that $f\circ\varphi_n$ is a test
  sequence~\cite{sela::dgog2}. Regard $\free_2$ as the fundamental
  group of a pair of pants. Then $f_n$ is
  always noninjective and is eventually non-$k$-pinching for any $k$.
\end{proof}

\begin{remark}
  Sela only defines test sequences for fundamental groups of closed surfaces.
  The construction is exactly the same in the case of a surface with boundary.
\end{remark}

\begin{question}
  Let $S$ be a surface with boundary and $\Gamma$ a hyperbolic
  group. Consider $\pi\colon S\to \Gamma$ and a sequence of
  automorphisms $\varphi_n$ of $\pi_1(S)$ such that
  $\pi\circ\varphi_n$ converges to $\pi_1(S)$. Then is
  $\pi\circ\varphi_n$ eventually non-$k$-pinching for any $k$ as long
  as $\pi$ embeds all subsurfaces $S'$ such that $\varphi_n$ stays in
  a finite set of cosets (up to conjugacy) of
  $\stab(\pi_1(S'))<\Aut(\pi_1(S))$?
\end{question}

\section{Surfaces without boundary}

Given an element of a group, it is natural to ask what the simplest
element in its normal closure is.

\begin{theorem*}[\cite{hempel}]
  Let $S$ be an orientable surface. If $g\in\pi_1(S)$ is not
  representable by a simple closed curve then $\ncl{g}$ contains no
  simple closed curves. 
\end{theorem*}

Conversely, in the category of groups which split over $\mathbb{Z}$,
there is no loss in considering only maps killing elements
supported on proper subsurfaces:

\begin{theorem}
  Let $S$ be a closed compact surface and suppose that $G$ is a
  quotient of $\pi_1(S)$ with a splitting over $\mathbb{Z}$. There is
  an intermediate quotient $\pi_1(S)\to G'\to G$ such that $G'$ is a
  quotient of $\pi_1(S)$ by a subgroup normally generated by an
  element supported on a proper subsurface of $S$.
\end{theorem}

\begin{proof}
  Suppose, without loss of generality, that $G=A*_CB$.  Let $X$ and
  $Y$ be spaces such that $\pi_1(Y)=A$, $\pi_1(Z)=B$. Represent $G$ as
  the fundamental group of a complex $X$ of the form
  \[X=((Y\sqcup Z)\sqcup(S^1\times I))/\{g(a)=(a,0),f(b)=(b,1)\}\]
  where $f,g\colon S^1\to Y,Z$ represent the inclusions $C\into A$ and
  $C\into B$. Represent $\pi_1(S)\to G$ by a continuous map
  $\varphi\colon S\to X$, and suppose that $\varphi$ is transverse to
  $S^1\times\frac{1}{2}$. We may assume that no connected component of
  $\Lambda=\varphi^{-1}(S^1\times\frac{1}{2})$ bounds a disk or a
  simple closed curve having trivial image in $X$, otherwise $\varphi$
  may be chosen to have smaller intersection with
  $S^1\times\frac{1}{2}$ or $\varphi$ kills a simple closed curve. Let
  $c$ be a curve in the kernel of $\varphi$ having smallest
  intersection number with $\Lambda$. Extend the map $c\to X$ to a
  disk $D$, arranging that the extension is also transverse to
  $S^1\times\frac{1}{2}$, and let $\alpha\colon I\to D$ be an
  outermost arc in the preimage of $S^1\times\frac{1}{2}$. Then
  $\alpha$ cuts off an arc $\gamma$ in $c$ which is sent to, without
  loss, $(Y\sqcup S^1\times(0,\frac{1}{2}))/g(a)=(a,0)$. Suppose that
  $\gamma$ connects distinct components $\lambda_1$ and $\lambda_2$ of
  $\Lambda$. Then $\left[\lambda_1,\gamma\lambda_2\gamma^{-1}\right]$
  has trivial image in $G$ and is contained in a proper subsurface of
  $S$. If $\gamma$ connects a component of $\Lambda$ to itself then
  either $c$ was not shortest or $\left[\gamma,\lambda\right]$ has
  trivial image in $G$.
\end{proof}

Thus if there are non-pinching quotients of surface groups that split
over $\zee$ then there are necessarily non-pinching quotients of
surface groups where relations are added to elements of a proper
subsurface only.

There are proper nonpinching limit group
quotients of surface groups with genus at least four:

\begin{example}
  \label{ex::hempel}
  Let $S$ be a closed orientable surface of genus $g\geq 4$ and let
  $P$ and $P'$ be two pairs of pants in $S$, each having three
  complementary components, such that $P\cap P'$ is a separating
  simple closed curve $c$. Let $f^{(\prime)}\colon P^{(\prime)}\to
  S^1$ be a map sending $c$ to $2$ and each other boundary component
  to $1$. Then the map $S\to S/f$ is not pinching, by Hempel's
  theorem, since the kernel is normally generated by a curve of self
  intersection number one. There is a natural map $S/f\to S/f\sqcup
  f'$. Then $S/f\cup f'$ is four (multi) handles glued along their
  boundary to a circle, and there is a retraction to the handle $H$ of
  smallest genus. The morphisms $S/f\to S/f\cup f'$ and $S/f\cup f'\to
  H$ are clearly both strict, and $\pi_1(S/f)$ is therefore a limit
  group. See Figure~\ref{hempelfigure}.
\end{example}

One can avoid Hempel's theorem and make this argument more in the spirit of Lemma~\ref{supernonpinching}. 

\begin{figure}[ht]
\labellist
\pinlabel $S$ at 52 69
\pinlabel $S/f$ at 175 69
\pinlabel $S/f\cup f'$ at 292 68
\pinlabel $H$ at 372 62
\endlabellist
\centerline{
\includegraphics[width=\textwidth]{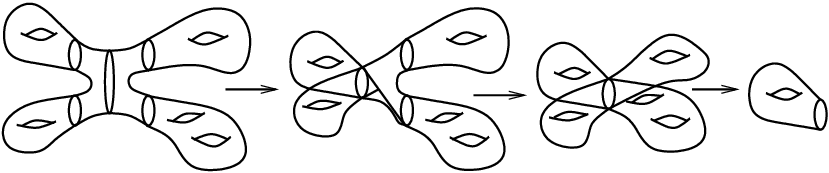}
}
\caption{Illustration for Example~\ref{ex::hempel}. We collapse the
  left-hand pair of pants to a circle so the common curve is mapped to
  $2$. The resulting group is a limit group since it has a strict
  homomorphism to a free group, represented here as the fundamental
  group of $H$. The quotient map $S\to S/f$ is nonpinching.}
\label{hempelfigure}
\end{figure}
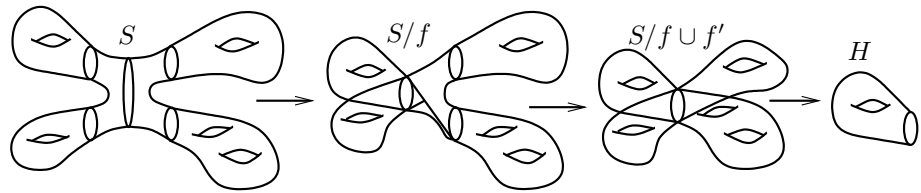

\subsection*{Assembling non-pinching maps}

The following criterion will be used to show that certain morphisms of
surface groups don't kill elements representable by curves of a given
self-intersection number.

\begin{lemma}
  Let $S$ be a compact surface with boundary and let $S_0\subset S$ be
  a compact deformation retract of $S$ such that $S\setminus S_0$ is a
  collection of half-open annuli. Let $Z$ be a space. If $f\colon
  S_0\to Z$ is $4k+4$ non-pinching then no non-boundary-parallel
  $k$-simple arc $\alpha\colon(I,\partial I)\to (S,\partial S)$ is
  homotopic rel boundary into $\partial S\subset S/f$.
\end{lemma}

\begin{proof}
  Let $p$ be a basepoint in $\partial S$ and let $\beta$ be an arc
  self intersecting only at $p$, representing the component of
  $\partial S$ containing $p$. We may assume that $\alpha$ is
  homotopic to an arc in $S$ starting and ending at $p$ with at most
  $k$ points of self intersection other than at $p$. If $\alpha$ is
  homotopic rel boundary into $\partial S$ then
  $\alpha\beta\alpha^{-1}\beta^{-1}$ has trivial image in $S/f$ and
  has self-intersection number at most $4k+4$.
\end{proof}

\begin{lemma}
  \label{supernonpinching}
  Let $S$ be a compact surface and let $S_0$ be a proper (not
  necessarily connected) closed essential subsurface. If $f\colon
  S_0\to Z$ is $4k+4$ non-pinching then the map $S\to S/f$ is
  non-$k$-pinching.
\end{lemma}

\begin{proof}
  Let $c$ be a $k$-simple closed curve in $S$, and homotope $c$ to
  have minimal intersection with $\partial S_0$. Let $S_0'$ be a
  regular neighborhood of $S_0$, and suppose that $c$ has minimal
  intersection with $\partial S_0'$ as well. Suppose that $c$ has
  trivial image in $S/f$. Let $D$ be a disk and choose a map $h\colon
  D\to S/f$ sending $\partial D$ to $c$. Arrange that $h$ is
  transverse to $\partial S_0'\subset S/f$. By the outermost arc
  argument some subarc $\alpha$ of $c$ mapping to $S_0'$ is homotopic
  into $\partial S_0'$. By the previous lemma, since $\alpha$ has at
  most $k$ points of self intersection it must have been homotopic
  into $\partial S_0'$, contradicting the assumption that $c$ have
  minimal intersection number with $\partial S_0$.
\end{proof}

\section{Proof of Theorem~\ref{maintheorem}}
\label{sec::proofmain}

The first bullet of Theorem~\ref{maintheorem} and
Corollary~\ref{cor:elementary} follow from:

\begin{example}
  Let $S_g$ be an orientable surface of genus at least three, and let
  $S'\subset S$ be a nonseparating four-holed sphere. Let $X$ be a
  graph, and suppose that $f\colon S'\to X$ has nonabelian image and
  doesn't kill any boundary components of $S'$. Then $\pi_1(S/f)$ is a
  limit group.  It suffices to show the case $g=3$: Let $S''\subset
  S\setminus S'$ be a $g-3$ handle. Let $q\colon S''\to p$ be the map
  to a point. The induced map $S/f\to (S/f)/q$ is strict, and $S/q$ is
  a genus three orientable surface, but $(S/f)/q=(S/q)/f$, identifying
  $S'\subset S$ with $S'\subset S/q$.  In genus $3$, let $r$ be the
  reflection of $S$ fixing $\partial S'$ pointwise. Define a map
  $h\colon S\to X$ by $x\mapsto f(r(x))$ for $x\in S\setminus S'$,
  otherwise $x\mapsto f(x)$. Then $h$ factors through $S/f$ and the
  induced map $S/f\to X$ is strict.  If $f$ is not $4k+4$-pinching
  then the natural map $S\to S/f$ is non-$k$-pinching by
  Lemma~\ref{supernonpinching}, and by Example~\ref{alphamaps} there
  are $4k+4$ non-pinching maps for any $k$.  By construction, $S/f$ is
  a tower over $X$, and Corollary~\ref{cor:elementary} holds by
  Theorem~\ref{sela::towers}.
\end{example}

For the second bullet we use the sequence of homomorphisms from a pair
of pants to the figure eight knot complement rather than maps to a
free or hyperbolic group.

\begin{example}
  \label{ex::genus2}
  Let $M$ be the fundamental group of the figure-eight knot
  complement.  Let $f_n$ be the sequence from
  Lemma~\ref{figureeightnonpinching}, and let $G_n$ be the double of
  $\pi_1(M)$ along the three elements $f_n(a)$, $f_n(b)$ and
  $f_n(ab)$. The double $G_n$ is fully residually $\pi_1(M)$ and there
  is a natural map $g_n\colon \pi_1(S_2)\onto G_n$ induced by
  $f_n$. The map $g_n$ is non-$k$-pinching for large $n$, again by
  Lemma~\ref{supernonpinching}. This proves the second bullet of
  Theorem~\ref{maintheorem}.
\end{example}

There are similar examples where $X$ or $M$ is replaced by a torsion
free hyperbolic group.

\begin{example}
  Let $S$ be the fundamental group of a closed oriented surface and
  let $S'\subset S$ be a nonseparating pair of pants in $S$.  Identify
  $\free_2$ with the fundamental group $S'$. Let $G$ be a torsion-free
  hyperbolic group admitting a nonpinching map from $\free_2$ with
  nonabelian image. Form the group $G_n=G*S/(P=f_n(P))$, where $f_n$
  is a sequence provided by Lemma~\ref{nonpinchingpairsofpants}. Then
  $G_n$ is a tower over $G$ as long as $\chi(S)<-2$. The natural map
  $\pi_1(S)\to G_n$ is noninjective and non-$k$-pinching for large
  $n$.
\end{example}

In Example~\ref{ex::genus2}, $M$ may certainly be replaced by a group
embedding in $\sl2c$ as in the last example. The sequence $f_n$ in
this case should be taken to be the less explicit one from
Lemma~\ref{nonpinchingpairsofpants}.

\section{Proof of Corollary~\ref{maintheorem2}}
\label{faithful}

In this section we show that Theorem~\ref{maintheorem} implies
Corollary~\ref{maintheorem2}. Everything here is standard and well
known. So far we have shown that every closed orientable surface
either admits a non-$k$-pinching quotient which is either a limit
group or a limit group over the figure-eight knot complement. All that
remains is to show that these quotients have faithful all-loxodromic
representations in $\sl2c$.

Recall that an element $g$ of $\sl2c$ is \emph{loxodromic} if
$\tr(g)^2\not\in\left[-2,2\right]\subset\mathbb{C}$, and an element
$\left[g\right]$ in $\psl2c$ is loxodromic if
$\tr(g)\not\in\left[0,4\right].$ The translation length $\Vert g\Vert$
of an element $g$ of $\sl2c$ is the infimum of displacements of points
in $\mathbb{H}^3$ under the action of $\overline g\in\psl2c$, and an
element has positive translation length if and only if it is
loxodromic.

\begin{lemma}
  \label{densitylemma}
  Let $V$ be an irreducible component of $\Hom(G,\sl2c)$. Fix $g\in
  G$. If $\rho(g)$ is loxodromic for some $\rho\in V$ then
  \[L_g\define\{\rho\in V\mid \rho(g)\text{ is loxodromic}\}\] 
  is open and dense in $V$.
\end{lemma}

\begin{proof}
  Let $\trace_g(\rho)=\trace(\rho(g))$. Then
  $L_g=\trace_g^{-1}(\mathbb{C}\setminus\left[-2,2\right])$. Since the
  trace and evaluation maps are polynomials, by the open mapping
  theorem, either $L_g$ is empty or its complement has empty interior.
\end{proof}

\begin{lemma}
  \label{faithfulrepresentations}
  Let $G$ be a group with a faithful all-loxodromic representation in
  $\sl2c$. If $H$ is a countable and fully-residually $G$ group then
  $H$ has a faithful all-loxodromic representation in $\sl2c$.
\end{lemma}

\begin{proof}
  Let $f_n\colon H\to G$ be a sequence of homomorphisms converging to
  $H$, and let $\rho\colon G\to\sl2c$ be a faithful all-loxodromic
  representation of $G$. We may assume, by passing to a subsequence,
  that all $\rho f_n$ are contained in the same irreducible component
  $V$ of $\Hom(H,\sl2c)$. Given a nontrivial $g\in H$ there is some
  $n$ such that $\rho(f_n(g))$ is loxodromic. By
  Lemma~\ref{densitylemma} $L_g$ is open and dense in $V$, and by
  Baire category theorem $W=\cap_{g\in H}L_g$ is nonempty. Any point
  in $W$ is a faithful all-loxodromic representation of $H$.
\end{proof}

Free groups have faithful all-loxodromic representations in $\sl2c$,
so Corollary~\ref{maintheorem2} follows from
Lemma~\ref{faithfulrepresentations} and the first bullet of
Theorem~\ref{maintheorem} when the genus is at least three. For genus
two see Example~\ref{ex::genus2} and the following.

\begin{lemma}
  The figure-eight knot complement has a faithful all-loxodromic
  representation in $\sl2c$.
\end{lemma}

\begin{proof}
  Let $M$ be the figure-eight knot complement. Let $N_k$ be the
  $(1,k)$ Dehn filling of $M$. The manifolds $N_k$ are hyperbolic and
  each arises as the quotient of $\mathbb{H}^3$ by $\Gamma_k$, where
  $\Gamma_k$ is some discrete cocompact subgroup of $\psl2c$. By
  Thurston (see~\cite[Theorem~3.1.1]{cs}) the representation
  $\pi_1(N_k)\to\psl2c$ lifts to a representation
  $\rho\colon\pi_1(N_k)\to\sl2c.$ The groups $\pi_1(N_k)$ converge
  algebraically to $\pi_1(M)$ and by Lemma~\ref{densitylemma} and
  Baire's theorem $\pi_1(M)$ has a faithful all-loxodromic
  representation in $\sl2c$.
\end{proof}

\bibliographystyle{amsalpha} \bibliography{slc}

\begin{flushleft}
Ann Arbor/Coventry\\
\emph{email:}\texttt{l.louder@ucl.ac.uk, lars@d503.net}
\end{flushleft}

\end{document}